\documentclass[11pt, twoside]{amsart}

\usepackage{amsmath,amssymb,amsthm,enumitem}
\usepackage{stmaryrd}

\usepackage{thm-restate,hyperref}

\usepackage[top=1.5in, bottom=1.5in, left=1.25in, right=1.25in]{geometry}

\usepackage{tikz-cd,tikz}

\allowdisplaybreaks

\theoremstyle{plain}
\newtheorem{theorem}{Theorem}[section]
\newtheorem{lemma}[theorem]{Lemma}
\newtheorem{cor}[theorem]{Corollary}
\newtheorem{prop}[theorem]{Proposition}
\newtheorem{conjecture}[theorem]{Conjecture}
\newtheorem*{lemma*}{Lemma}
\newtheorem*{cor*}{Corollary}
\newtheorem*{theorem*}{Theorem}

\theoremstyle{definition}

\newtheorem{problem*}{Problem}
\newtheorem{example}{Example}
\newtheorem{definition}[theorem]{Definition}

\theoremstyle{remark}

\newtheorem*{fact*}{Fact}
\newtheorem*{remark}{Remark}

\newtheorem*{notation}{Notation}
\newtheorem*{convention}{Convention}

\let\oldproofname=\proofname
\renewcommand{\proofname}{\rm\bf{\oldproofname}}

\newcommand{\refthm}[1]{Theorem \ref{#1}}
\newcommand{\refprop}[1]{Proposition \ref{#1}}
\newcommand{\refcor}[1]{Corollary \ref{#1}}
\newcommand{\refconj}[1]{Conjecture \ref{#1}}
\newcommand{\reflem}[1]{Lemma \ref{#1}}
\newcommand{\refdef}[1]{Definition \ref{#1}}
\newcommand{\refsec}[1]{Section \ref{#1}}

\newcommand{\refeq}[1]{(\ref{#1})}

\newcommand{\Z}{\mathbb Z}

\newcommand{\phii}{\varphi}

\makeatletter
\providecommand*{\twoheadrightarrowfill@}{%
  \arrowfill@\relbar\relbar\twoheadrightarrow
}
\providecommand*{\twoheadleftarrowfill@}{%
  \arrowfill@\twoheadleftarrow\relbar\relbar
}
\providecommand*{\xtwoheadrightarrow}[2][]{%
  \ext@arrow 0579\twoheadrightarrowfill@{#1}{#2}%
}
\providecommand*{\xtwoheadleftarrow}[2][]{%
  \ext@arrow 5097\twoheadleftarrowfill@{#1}{#2}%
}
\makeatother

\newcommand{\surj}{\twoheadrightarrow}

\newcommand{\inj}{\hookrightarrow}

\newcommand{\op}[1]{\operatorname{{#1}}}

\newcommand{\mc}[1]{\mathcal{{#1}}}

\begin{document}
    \title{A Spectral Sequence for Dehn Fillings}
    \author{Oliver H. Wang}
        \address{Department of Mathematics, University of Chicago, Chicago, IL 60637}
    \email{oliverwang@uchicago.edu}
    \begin{abstract}
		We study how the cohomology of a type $F_\infty$ relatively hyperbolic group pair $(G,\mc{P})$ changes under Dehn fillings (i.e. quotients of group pairs).
		For sufficiently long Dehn fillings where the quotient pair $(\bar{G},\bar{\mc{P}})$ is of type $F_\infty$, we show that there is a spectral sequence relating the cohomology groups $H^i(G,\mc{P};\Z G)$ and $H^i\left(\bar{G},\bar{\mc{P}};\Z\bar{G}\right)$.
		As a consequence, we show that essential cohomological dimension does not increase under these Dehn fillings.
	\end{abstract}
    \maketitle

\section{Introduction}
Classically, Dehn filling is a process by which one can obtain infinitely many distinct closed hyperbolic $3$-manifolds from a hyperbolic $3$-manifold with toroidal cusps.
This notion has been adapted to relatively hyperbolic groups by Osin in \cite{Osin07} and by Groves and Manning \cite{GM08}.
In these works, the authors introduce a group theoretic analogue of Dehn filling and show that this process, when applied to a relatively hyperbolic group, yields a relatively hyperbolic group.
We study how the cohomology groups $H^i(G,\mc{P};\Z G)$ behave under this procedure which, by a result of Manning and the author in \cite{MW}, are isomorphic to the cohomology of the Bowditch boundary.

Our motivating example is a theorem by Groves, Manning and Sisto.
In \cite{GMS}, it is shown that, if $(G,\mc{P})$ is relatively hyperbolic with Bowditch boundary $\partial(G,\mc{P})\cong S^2$ and such that each subgroup of $\mc{P}$ is free abelian of rank $2$, then for all sufficiently long Dehn fillings where the quotients $P_i/N_i$ are virtually infinite cyclic, the quotient group $\bar{G}$ is hyperbolic with boundary $\partial G\cong S^2$.
Assuming that all groups involved are type $F_\infty$, the following isomorphisms can be derived from Bieri and Eckmann's work on relative cohomology (\cite[Proposition 1.1]{BE}), Bestvina and Mess's results on the cohomology of hyperbolic groups (\cite[Corollary 1.3]{BM91}) and the aforementioned theorem \cite[Theorem 1.1]{MW}.
\begin{equation}\label{eq: GMS example}
H^i(G,\mc{P};\Z G)\cong\begin{cases}\Z&i=3\\0&i\neq3\end{cases}\hspace{1 cm}H^i(\bar{G},\bar{\mc{P}};\Z\bar{G})\cong\begin{cases}\Z&i=3\\\bigoplus_{\bar{P}\in\bar{\mc{P}}}\bigoplus_{\bar{g}\bar{P}\in\bar{G}/\bar{P}}\Z&i=2\\0&i\neq2,3\end{cases}
\end{equation}
From this example, it would seem that there is only hope for a statement about top dimensional cohomology.
Our main result, stated below, shows that we can relate the cohomology of a relatively hyperbolic group and the cohomology of its Dehn fillings in all dimensions.

\begin{restatable}{theorem}{MainThm}\label{thm: Main Thm}
Suppose $(G,\mc{P})$ is a type $F_\infty$ relatively hyperbolic group pair and let $n=\op{ecd}(G,\mc{P})$.
For all sufficiently long $F_\infty$ Dehn fillings $(\bar{G},\bar{\mc{P}})$, there is a spectral sequence
\[
E^2_{pq}=H_p(K;H^{n-q}(G,\mc{P};\Z G))\Rightarrow H^{n-(p+q)}\left(\bar{G},\bar{\mc{P}};\Z\bar{G}\right)
\]
where $K=\op{ker}(G\surj\bar{G})$.
The differentials $d^r_{pq}$ have bidegree $(-r,r-1)$.
\end{restatable}

Here, $\op{ecd}$ stands for ``essential cohomological dimension'' which we define to be the largest $n$ such that $H^n(G,\mc{P};\Z G)$ is nonzero.
In general, this is different from cohomological dimension.
This is discussed in more detail in \refsec{sec: ecd}.

Returning to the example in \cite{GMS}, the kernel $K$ is an infinitely generated free group by \cite[Theorem 4.8]{GMS}.
\refthm{thm: Main Thm} implies that $H^2(\bar{G},\bar{\mc{P}};\Z\bar{G})\cong H_1(K;\Z)$.
If $\Z$ has a trivial $K$-action, then the calculation in \refeq{eq: GMS example} follows immediately.

We now mention some immediate applications of \refthm{thm: Main Thm}.

\subsection{The Bowditch Boundary}
As mentioned above, the motivation for studying these cohomology groups is to study the Bowditch boundary, which was originally introduced by Bowditch in \cite{Bow}
By \cite[Theorem 1.1]{MW}, $H^i\left(G,\mc{P};\Z G\right)\cong\check{H}^{i-1}\left(\partial\left(G,\mc{P}\right);\Z\right)$ whenever $(G,\mc{P})$ is of type $F_\infty$ (here, the right hand side is reduced \v{C}ech cohomology).
This is an isomorphism of $\Z G$-modules so we can replace the cohomology groups of the pairs with the \v{C}ech cohomology groups of the boundaries in \refthm{thm: Main Thm}.

\begin{cor}\label{cor: spectral sequence with cohom of boundary}
Suppose $(G,\mc{P})$ is a type $F_\infty$ relatively hyperbolic group pair and let $n=\op{ecd}(G,\mc{P})$.
For all sufficiently long $F_\infty$ Dehn fillings $(\bar{G},\bar{\mc{P}})$, there is a spectral sequence
\[
E^2_{pq}=H_p\left(K;\check{H}^{n-q-1}\left(\partial\left(G,\mc{P}\right);\Z\right)\right)\Rightarrow\check{H}^{n-(p+q)-1}\left(\partial\left(\bar{G},\bar{\mc{P}}\right);\Z\right)
\]
where $K=\op{ker}(G\surj\bar{G})$ and $\check{H}^i$ denotes reduced \v{C}ech cohomology.
The differentials $d^r_{pq}$ have bidegree $(-r,r-1)$.
\end{cor}

For a more concrete topological application of \refthm{thm: Main Thm}, consider the following conjecture, made in \cite{MW}.

\begin{conjecture}\label{conj: dim boundary conjecture}
If $(G,\mc{P})$ is a type $F$ relatively hyperbolic group pair, then $\dim\partial(G,\mc{P})=\op{cd}(G,\mc{P})-1$.
\end{conjecture}

Of course, the case when $(G,\mc{P})$ is of type $F_\infty$ is also of interest but $\partial(G,\mc{P})$ is not a $Z$-set in general, which makes topological statements about $\partial(G,\mc{P})$ more difficult to prove.
The conjecture is shown to be true in the case when $(G,\mc{P})$ is of type $F$ and $\op{cd}G<\op{cd}(G,\mc{P})$.
Moreover, \cite[Corollary 1.4]{BM91} gives the non-relatively hyperbolic statement of this conjecture: whenever $G$ is of type $F$ and hyperbolic, $\dim\partial G=\op{cd}(G)-1$.
This motivates the question of how $\op{cd}$ behaves under type $F$ Dehn fillings of type $F$ pairs and, more generally, how $\op{ecd}$ behaves under $F_\infty$ Dehn fillings.
Our spectral sequence provides a partial answer for sufficiently long fillings.

\begin{restatable}{cor}{DimPreserve}\label{cor: Dim Preserve}
Suppose $(G,\mc{P})$ is a type $F_\infty$ relatively hyperbolic group pair.
Let $n=\op{ecd}(G,\mc{P})$.
Then, for all sufficiently long $F_\infty$ Dehn fillings,  $H^n(\bar{G},\bar{\mc{P}};\Z\bar{G})\cong H_0(K;H^n(G,\mc{P};\Z G))$ and $H^N(\bar{G},\bar{\mc{P}};\Z\bar{G})=0$ for $N>n$.
In particular, if $H^n(G,\mc{P};\Z G)\cong\Z$, then $\op{ecd}(G,\mc{P})=\op{ecd}(\bar{G},\bar{\mc{P}})$.
\end{restatable}
\begin{proof}
In the spectral sequence of \refthm{thm: Main Thm}, the second page has nonzero terms only in the first quadrant.
So $E^2_{0,0}=E^\infty_{0,0}$ and convergence implies that there is the isomorphism $H^n(\bar{G},\bar{\mc{P}};\Z\bar{G})\cong H_0(K;H^n(G,\mc{P};\Z G))$.
Additionally, $E^2_{pq}=0$ when $p+q<0$ so $H^N(\bar{G},\bar{\mc{P}};\Z\bar{G})=0$ when $N>n$.
\end{proof}

\subsection{Duality Pairs}
Duality pairs, studied by Bieri and Eckmann in \cite{BE}, are pairs whose homology and cohomology admit cap product isomorphisms resembling Poincar\'e duality.
One can also think of such pairs as those whose cohomology groups $H^i(G,\mc{P};\Z G)$ behave nicely.
Specifically, if $(G,\mc{P})$ is a duality pair, then it has finite cohomological dimension and $H^i(G,\mc{P};\Z G)=0$ for all $i\neq\op{cd}(G,\mc{P})$.
So when $(G,\mc{P})$ is a duality pair, the spectral sequence degenerates on the second page; everything outside the row $E^2_{*0}$ vanishes and \refthm{thm: Main Thm} simplifies into the following statement.

\begin{restatable}{cor}{ClassicalCase}\label{cor: Duality Case}
Suppose $(G,\mc{P})$ is a relatively hyperbolic duality pair and let $n=\op{cd}(G,\mc{P})$.
Then, for all sufficiently long $F_\infty$ Dehn fillings $(\bar{G},\bar{\mc{P}})$,
\[
H^{n-i}\left(\bar{G},\bar{\mc{P}};\Z\bar{G}\right)\cong H_i(K;H^n(G,\mc{P};\Z G)).
\]
\end{restatable}

\subsection{Generalizations}
Many of the results in this paper on relatively hyperbolic groups can be generalized to the case of $(G,\mc{P})$ where $\mc{P}$ is a hyperbolically embedded family of subgroups (see \cite{DGO} for the definition and properties).
In this general case, however, we do not have a geometric interpretation of $H^n(G,\mc{P};\Z G)$.
Nonetheless we will include these more general statements in case they become useful for future works.
These statements can be safely ignored by readers interested only in relatively hyperbolic groups.

\subsection{Outline}
We begin by introducing some definitions and summarizing some results of \cite{MW} in \refsec{sec: prelim}.
The space $\mc{X}(G,\mc{P})$ is introduced in \refdef{def: contractible cusped space} which will allow us to use topological tools in order to study $H^i(G,\mc{P};\Z G)$.
In \refsec{sec: Dehn Filling}, we review Dehn fillings of group pairs.
The main result here is \refprop{prop: homotopy of cusped space/K}.
This section contains most of the geometry used in this paper.
In \refsec{sec: homalg}, we use homological algebra to prove \refthm{thm: Main Thm}.

\subsection{Acknowledgements}
The author would like to thank Jason Manning for helpful conversations and for pointing out an error in a very early version of this work.
The author would also like to thank Bin Sun for his comments on a previous version of this work.
Finally, the author would like to thank the referees for their comments and their patience.
The author was partially supported by the National Science Foundation, grant DMS-1462263.

\section{Preliminaries}\label{sec: prelim}

\subsection{Cohomology with Compact Support}
We briefly introduce some results on cohomology with compact support that will be needed later.
We refer to \cite[Appendix]{MW} for a more detailed discussion.

\begin{theorem}\label{thm: LES cohom compact supp}
Suppose $X$ is locally contractible, locally compact Hausdorff space and let $V$ be a locally contractible, closed subspace with locally contractible complement.
Let $A$ be an abelian group.
Then, there is the following long exact sequence where $\iota^*$ is induced by the inclusion $\iota:F\inj X$.
\[
\cdots H_c^i(X\setminus V;A)\rightarrow H_c^i(X;A)\rightarrow H_c^i(V;A)\rightarrow H_c^{i+1}(X\setminus V;A)\rightarrow\cdots
\]
\end{theorem}

\begin{cor}\label{cor: prod with half open interval}
If $X$ is a locally contractible, locally compact Hausdorff space, then $H_c^i([0,1)\times X;A)=0$ for all $i\ge0$ and all coefficients $A$.
\end{cor}

\begin{proof}
This follows from \refthm{thm: LES cohom compact supp} and the fact that $\{1\}\times X\inj[0,1]\times X$ is a proper homotopy equivalence.
\end{proof}

\subsection{Relative Cohomology}
Here, we introduce some definitions and results of \cite{BE} on relative group cohomology and we summarize some results in \cite{MW}.

\begin{definition}\label{def: group pair, comp gen set}
A \emph{group pair} $(G,\mc{P})$ is a group $G$ and a finite set of subgroups $\mc{P}=\{P_i\}_{i\in I}$.
\end{definition}

\begin{definition}\label{def: type F, F infty}
$G$ is \emph{of type $F$} (resp. \emph{of type $F_\infty$}) if it has a classifying space with finitely many cells (resp. finitely many cells in each dimension).
A group pair $(G,\mc{P})$ is \emph{of type $F$} (resp. \emph{of type $F_\infty$}) if $G$ and each $P_i\in\mc{P}$ are of type $F$ (resp. of type $F_\infty$).
\end{definition}

\begin{definition}\label{def: hom and cohom of group pair}
For a group pair $(G,\mc{P})$, let $\Z G/\mc{P}:=\oplus_{P\in\mc{P}}\Z G/P$ and let $\epsilon:\Z G/\mc{P}\rightarrow\Z$ be the augmentation map.
Let $\Delta:=\ker(\epsilon)$.
Then for any $\Z G$-module $M$, the \emph{relative homology} and \emph{relative cohomology} groups are
\[
H_i(G,\mc{P};M):=\op{Tor}_{i-1}^G(\Delta,M)\hspace{1 cm}H^i(G,\mc{P};M):=\op{Ext}_G^{i-1}(\Delta,M)
\]
\end{definition}

Letting $H_i(\mc{P};M):=\oplus_{P\in\mc{P}}H_i(P;M)$ and $H^i(\mc{P};M):=\prod_{P\in\mc{P}}H^i(P;M)$, there are long exact sequences
\[
...\rightarrow H_i(\mc{P};M)\rightarrow H_i(G;M)\rightarrow H_i(G,\mc{P};M)\rightarrow H_{i-1}(\mc{P};M)\rightarrow...
\]
\[
...\rightarrow H^i(G,\mc{P};M)\rightarrow H^i(G;M)\rightarrow H^i(\mc{P};M)\rightarrow H^{i+1}(G,\mc{P};M)\rightarrow...
\]

\begin{definition}\label{def: Eilenberg MacLane pair}
If $G$ is a group, we will use $BG$ to denote its classifying space and $EG$ to denote the universal cover of $BG$.
These spaces will be subject to the finiteness conditions imposed on $G$ (for instance, if $G$ is of type $F$, we will assume $BG$ is a finite CW complex).
Let $(G,\mc{P})$ be a group pair.
An \emph{Eilenberg-MacLane pair for $(G,\mc{P})$} is a pair $(X,Y)$ where $X=BG$ and $Y=\sqcup_{P\in\mc{P}}BP$ is a subspace of $X$ such that the inclusion $BP\inj X$ induces the inclusion $P\inj G$ (after a suitable choice of path between basepoints).
\end{definition}

\subsection{Cusped Spaces}

\begin{definition}\label{def: contractible cusped space}
Let $(G,\mc{P})$ be a type $F_\infty$ group pair where $\mc{P}$ is finite and let $(X,Y)$ be an Eilenberg-MacLane pair for $(G,\mc{P})$ such that $X$ has finitely many cells in each dimension.
Let $\tilde{X}$ be the universal cover of $X$ and let $\tilde{Y}\subseteq\tilde{X}$ be the preimage of $Y$ under the projection $\tilde{X}\surj X$.
Then, the \emph{contractible cusped space for $(G,\mc{P})$} is
\[
\mc{X}(G,\mc{P}):=\tilde{X}\cup\left([0,\infty)\times\tilde{Y}\right)
\]
where $\{0\}\times\tilde{Y}$ is identified with $\tilde{Y}\subseteq\tilde{X}$.

The \emph{$N$-connected cusped space for $(G,\mc{P})$} is the subspace
\[
\tilde{X}^{(N+1)}\bigcup\left([0,\infty)\times\tilde{Y}^{(N+1)}\right)
\]
which we denote $\mc{X}(G,\mc{P},N)$.
Here, if $Z$ is a CW complex, $Z^{(N)}$ denotes the $N$-skeleton of $Z$.
\end{definition}


\begin{remark}
In \cite{MW}, $\mc{X}(G,\mc{P},N)$ is equipped with a metric.
For our purposes, it suffices to consider $\mc{X}(G,\mc{P},N)$ as a CW-complex.
\end{remark}

In \cite{MW}, the following result was proved.
\begin{prop}\label{prop: cohom of N-conn cusped spaces}
For a type $F_\infty$ group pair $(G,\mc{P})$ and all $N\ge0$, $H^i(G,\mc{P};\Z G)\cong H_c^i(\mc{X}(G,\mc{P},N);\Z)$ for $i\le N$.
\end{prop}

Using these spaces, \cite{MW} proves the following theorem about relatively hyperbolic groups.

\begin{theorem}\label{thm: MW main theorem}
Let $(G,\mc{P})$ be a type $F_\infty$ relatively hyperbolic group pair.
Then, $H^i(G,\mc{P};\Z G)\cong\check{H}^{i-1}(\partial(G,\mc{P});\Z)$ where the right hand side is reduced \v{C}ech cohomology.
\end{theorem}

Since $\mc{X}(G,\mc{P})$ is not necessarily a locally compact space, it does not make sense to discuss compactly supported cohomology.
Because we assume $\mc{P}$ is finite and all groups are type $F_\infty$, this space has locally finite skeleta in each dimension.
So it has a cochain complex of compactly supported cellular cochains.
Thus we can make the following definition.

\begin{definition}\label{def: compactly supp cohom for contractible cusped space}
Suppose $X$ is a CW complex such that each skeleton is locally finite.
Let $C_c^*(X;A)$ be the cochain complex of compactly supported cellular cochains on $X$ with coefficients in an abelian group $A$.
Define
\[
H_c^i(X;A):=H^i(C_c^*(X;A)).
\]
\end{definition}

\begin{prop}\label{prop: cohom contractible cusped space}
Suppose $(G,\mc{P})$ is a type $F_\infty$ group pair.
Then for all $i$, $H^i(G,\mc{P};\Z G)\cong H_c^i(\mc{X}(G,\mc{P});\Z)$.
\end{prop}
\begin{proof}
By considering the compactly supported cellular cochains, we have $H_c^i(\mc{X}(G,\mc{P});\Z)\cong H_c^i(\mc{X}(G,\mc{P},N);\Z)$ for $i\le N$.
By \refprop{prop: cohom of N-conn cusped spaces}, the right hand side is $H^i(G,\mc{P};\Z G)$ for $i\le N$.
The result follows from taking $N$ arbitrarily large.
\end{proof}

\subsection{Essential Cohomological Dimension}\label{sec: ecd}
\begin{definition}\label{def: ecd}
Let $G$ be a group.
Then, the \emph{essential cohomological dimension of $G$} is
\[
\op{ecd}G=\sup\{n\ge0|H^n(G;\Z G)\neq0\}.
\]
If $(G,\mc{P})$ is a group pair, then the \emph{essential cohomological dimension of $(G,\mc{P})$} is
\[
\op{ecd}(G,\mc{P})=\sup\{n\ge1|H^n(G,\mc{P};\Z G)\neq0\}.
\]
\end{definition}

\begin{example}
If $H$ is a finite index subgroup of $G$, then $H^n(G;\Z G)\cong H^n(H;\Z H)$ for all $n$.
So, $\op{ecd}$ is invariant under commensurability.
In particular, if $G$ is a finite group, $\op{cd}(G)$ is infinite but $\op{ecd}(G)=0$.
\end{example}

\begin{remark}
In the case when $G$ is of type $F_\infty$ with finite cohomological dimension, \cite[Proposition VIII.6.1]{Brown} and \cite[Proposition VIII.6.7]{Brown} say that $\op{cd}G=\op{ecd}G$.
More generally, if $G$ is of type $F_\infty$ with finite virtual cohomological dimension, then $\op{vcd}G=\op{ecd}G$.
\end{remark}

\begin{example}
If $G$ is a torsion free hyperbolic group, then $\op{cd}(G)=\op{ecd}(G)$.
\end{example}

\begin{remark}
Note that, the essential cohomological dimension of a pair is defined to be at least $1$, rather than $0$.
This is reasonable because relative group cohomology is defined with a dimension shift so $H^0(G,\mc{P};\Z G)$ does not have any meaning.
Moreover, \refconj{conj: dim boundary conjecture} would not be true if we allowed the essential cohomological dimension to be $0$.
Since $H^i(G,\{G\};\Z G)=0$ for all $i$ and $\partial(G,\{G\})$ is a point, we would have $\dim\partial(G,\{G\})=0$ and $\op{ecd}(G,\{G\})-1=-1$.
\end{remark}

For type $F_\infty$ relatively hyperbolic group pairs, we have the following result.

\begin{prop}\label{prop: ecd is finite}
If $(G,\mc{P})$ is a type $F_\infty$ relatively hyperbolic group pair, then $\op{ecd}(G,\mc{P})$ is finite.
\end{prop}
\begin{proof}
This follows from \refthm{thm: MW main theorem} and the fact that $\partial(G,\mc{P})$ is finite dimensional.
\end{proof}

\begin{remark}
The author is not aware of an example of a group with infinite essential cohomological dimension.
\end{remark}

\subsection{Relatively Hyperbolic Groups}
In this subsection, we will briefly discuss relatively hyperbolic groups.
As we will import the necessary geometry rather than working with it explicitly, we will not give rigorous definitions here.

Suppose $(G,\mc{P})$ is a type $F_\infty$ group pair.
As mentioned previously, $\mc{X}(G,\mc{P},N)$ is equipped with a metric $d_N$ in \cite{MW}.
It is shown that $(\mc{X}(G,\mc{P},N),d_N)$ is quasi-isometric to the cusped space of \cite{GM08}.
Thus, $(G,\mc{P})$ is \emph{relatively hyperbolic} if and only if $(\mc{X}(G,\mc{P},N),d_N)$ is $\delta$-hyperbolic for some $\delta>0$.
If $(G,\mc{P})$ is relatively hyperbolic, the Gromov boundary of $\mc{X}(G,\mc{P},N)$ is the \emph{Bowditch boundary of $(G,\mc{P})$} and is denoted by $\partial(G,\mc{P})$.

\begin{remark}
We are only considering pairs of type $F_\infty$ in this paper.
For a more comprehensive and thorough description of relative hyperbolicity, see \cite{Hruska}.
\end{remark}

\begin{example}
Suppose that $G=\pi_1(M)$ where $M$ is a complete, finite volume hyperbolic $n$-manifold.
Let $\mc{P}$ be the collection of cusp subgroups of $M$.
Then, $(G,\mc{P})$ is relatively hyperbolic and $\partial(G,\mc{P})\cong S^{n-1}$.
\end{example}

\section{Dehn Filling}\label{sec: Dehn Filling}

\begin{definition}\label{def: Dehn Filling}
Given a group pair $(G,\mc{P})$, where $\mc{P}=\{P_1,...,P_m\}$, a \emph{Dehn filling} of $(G,\mc{P})$ is a pair $(G/K,\bar{\mc{P}})$ where $K$ is the normal closure in $G$ of normal subgroups $N_i\trianglelefteq P_i$ and $\bar{\mc{P}}$ is the family of images of $\mc{P}$ under this quotient.
The $N_i$ are called the \emph{filling kernels}.
\end{definition}

\begin{notation}
We will henceforth write $\bar{G}$ for $G/K$ and $\bar{P}_i$ for the image of $P_i$ under the quotient.
If we wish to emphasize the filling kernels, we write $G(N_1,...,N_m)$ for $G/K$.
\end{notation}

\begin{definition}\label{def: suff long Dehn filling}
Following \cite{GMS}, we say that a property holds for \emph{sufficiently long Dehn fillings} if there is a finite set $\mc{B}\subseteq G\setminus\{1\}$ such that the property holds for all $(G(N_1,...,N_m),\bar{\mc{P}})$ where $N_i\cap\mc{B}=\emptyset$ for each $i$.
\end{definition}

\begin{remark}
A reader unfamiliar with the above definitions can think of ``Dehn fillings'' as ``quotient of group pairs'' and ``sufficiently long Dehn fillings'' as ``most quotients of group pairs.''
\end{remark}

\begin{definition}\label{def: F infty Dehn filling}
Let $(G,\mc{P})$ be a group pair.
Then $(\bar{G},\bar{\mc{P}})$ is an \emph{$F_\infty$ Dehn filling} if $(\bar{G},\bar{\mc{P}})$ is a type $F_\infty$ group pair.
\end{definition}

To study the kernel of Dehn fillings, we use the Cohen-Lyndon property, which is studied by Sun in \cite{Sun}.

\begin{definition}\label{def: cohen lyndon property}
Let $G$ be a group and let $\{P_\lambda\}_{\lambda\in\Lambda}$ be a collection of subgroups.
Suppose that $\{N_\lambda\}_{\lambda\in\Lambda}$ is a collection of subgroups such that $N_\lambda\trianglelefteq P_\lambda$.
Let $K$ denote the normal closure of $\cup_{\lambda\in\Lambda}N_\lambda$ in $G$.
Then, the triple $(G,\{P_\lambda\}_{\lambda\in\Lambda},\{N_\lambda\}_{\lambda\in\Lambda})$ is said to have the \emph{Cohen-Lyndon property} if, for each $\lambda\in\Lambda$, there is a left transversal $T_\lambda$ of $P_\lambda K$ in $G$ such that
\[
K=*_{\lambda\in\Lambda}*_{t\in T_\lambda}N_\lambda^t.
\]
In this case, we say that $T_\lambda$ is the \emph{transversal for the triple $(G,\{P_\lambda\}_{\lambda\in\Lambda},\{N_\lambda\}_{\lambda\in\Lambda})$}.
\end{definition}

The following result in \cite{GMS} implies that sufficiently long Dehn fillings in relatively hyperbolic groups satisfy the Cohen-Lyndon property.

\begin{theorem}\label{thm: ker of Dehn filling}
Suppose $(G,\{P_1,...,P_m\})$ is relatively hyperbolic where $G$ and each $P_i$ are finitely generated.
Let $\mc{C}$ be the set of parabolic points of $\partial(G,\mc{P})$.
For a Dehn filling $\bar{G}=G(N_1,...,N_m)$, let $K=\op{ker}(G\surj\bar{G})$.
Then, for sufficiently long Dehn fillings, there is a subset $T\subseteq\mc{C}$ intersecting each $K$-orbit exactly once such that
\[
K=*_{t\in T}\left(K\cap\op{Stab}t\right)
\]
and each $K\cap\op{Stab}t$ is conjugate to some filling kernel $N_i$.
\end{theorem}

\begin{cor}\label{cor: rel hyp CL}
Let $(G,\{P_1,...,P_m\})$ be as in \refthm{thm: ker of Dehn filling}.
Then sufficiently long Dehn fillings of $(G,\{P_1,...,P_m\})$ satisfy the Cohen-Lyndon property.
\end{cor}
\begin{proof}
Consider $\mc{X}(G,\mc{P},N)$ with the metric in \cite{MW}.
Every parabolic point is the boundary point of some $g\cdot \left([0,\infty)\times EP_i^{(N+1)}\right)\subseteq\mc{X}(G,\mc{P},N)$.
This follows from \cite[Section 2.5]{GMS} and the quasi-isometry from the Groves-Manning combinatorial cusped space to $\mc{X}(G,\mc{P},N)$ (see \cite[Proposition 3.10]{MW}).
Thus, if $c_i\in\mc{C}$ is the parabolic point with stabilizer $P_i$, then $\mc{C}=\sqcup_{i=1}^m\sqcup_{gP_i\in G/P_i}g\cdot c_i=\sqcup_{i=1}^m G/P_i$ as $G$-sets.
The set of $K$-orbits is $\sqcup_{i=1}^m K\backslash G/P_i=\sqcup_{i=1}^m G/KP_i$.
By \refthm{thm: ker of Dehn filling} above, there is a set $T\subseteq\mc{C}$ which is a transversal for the triple $(G,\{P_1,...,P_m\},\{N_1,...,N_m\})$ and satisfies the desired property.
\end{proof}

Although \refcor{cor: rel hyp CL} suffices to prove \refthm{thm: Main Thm}, enough machinery has been developed to prove a generalization.
In particular, there is the following result of \cite{Sun}.
\begin{theorem}\label{thm: LC hyp emb}
Suppose $G$ is a group with with $\{H_\lambda\}_{\lambda\in\Lambda}$ a hyperbolically embedded family of subgroups.
Then for sufficiently long Dehn fillings with filling kernels $N_\lambda\trianglelefteq H_\lambda$, the triple $(G,\{H_\lambda\}_{\lambda\in\Lambda},\{N_\lambda\}_{\lambda\in\Lambda})$ satisfies the Cohen-Lyndon property.
\end{theorem}

Suppose $(G,\mc{P})$ is a group pair and that $\left(\bar{G},\bar{\mc{P}}\right)$ is a Dehn filling satisfying the Cohen-Lyndon property and let $K$ denote the kernel of $G\rightarrow\bar{G}$.
It is immediate from the definition that, if $X$ is a classifying space for $G$ with universal cover $\tilde{X}$, then $\tilde{X}/K$ is homotopy equivalent to a wedge of classifying spaces for $N_i$.
We will need the following slightly stronger statement.

\begin{prop}\label{prop: homotopy of cusped space/K}
Suppose $\{P_1,...,P_m\}$ is a hyperbolically embedded family of subgroups of $G$.
Let $(X,Y)$ be an Eilenberg-MacLane pair for $(G,\{P_1,...,P_m\})$.
Let $\tilde{X}$ be the universal cover of $X$ and let $\tilde{Y}$ be the preimage of $Y$ under $\tilde{X}\surj X$.
If $(\bar{G},\bar{P})$ is a Dehn filling given by $\{N_1,...,N_m\}$, let $K=\ker(G\surj\bar{G})$
Define a space $V$ to be a vertex with an edge to each component of $\tilde{Y}/K$.
Then for sufficiently long Dehn fillings, there is a homotopy equivalence $f\colon V\rightarrow \tilde{X}/K$ such that $f$ restricts to the identity on $\tilde{Y}/K$.
\end{prop}
\begin{proof}
We may assume the conclusion of \refthm{thm: LC hyp emb}.
By \cite[Theorem 7.17a]{DGO}, we may also assume that $K\cap P_i=N_i$ (in the relatively hyperbolic setting, one may instead use \cite[Theorem 1.1]{Osin07}).

First, we establish some temporary notation.
For a path-connected space $V$, let $V_*$ denote $V$ with a disjoint basepoint and an edge connecting the two components.
We will also assume $X,\tilde{X}$ and $\tilde{X}/K$ are pointed such that the covering maps are basepoint preserving.
For each $i=1,...,m$, let $T_i$ be a transversal as in \refdef{def: cohen lyndon property}.


Now, we identify the set $\sqcup_{i=1}^mT_i$ with the components of $\tilde{Y}/K$ as follows.
Write $\tilde{Y}=\sqcup_{i=1}^m\sqcup_{gP_i\in G/P_i}g\cdot EP_i$ and consider the $K$ action on $\tilde{Y}$.
If for some $k\in K$, $kg\cdot EP_i=g\cdot EP_i$, then $g^{-1}kg\in P_i\cap K=N_i$ so the components of the quotient are the spaces $(EP_i)/N_i$.
It follows that $\tilde{Y}/K=\sqcup_{i=1}^m\sqcup_{\bar{g}\bar{P}_i\in\bar{G}/\bar{P}_i}\bar{g}\cdot EP_i/N_i$.
Associate to an element of $t\in T_i$ the component $\bar{t}\cdot EP_i/N_i$.
Since $T_i$ is a transversal, this gives the desired bijective correspondence.

Let $Y_t$ denote the component of $\tilde{Y}/K$ corresponding to $t\in T_i$.
We claim that there is a basepoint preserving map $(Y_t)_*\rightarrow\tilde{X}/K$ restricting to the inclusion on $Y_t$ and inducing the inclusion $N_i^t\inj K$ on fundamental groups.

We define the map as follows.
Since $(X,Y)$ is an Eilenberg-MacLane pair, there is a map $\psi:(BP_i)_*\rightarrow X$ inducing $P_i\inj G$ and restricting to the inclusion on $BP_i$.
Let $\sigma:[0,1]\rightarrow X$ denote the restriction of this map on the added edge.
Let $\tau:[0,1]\rightarrow X$ be a loop based at the basepoint of $X$ representing $t$.
Then, the map $\phii:(BP_i)_*\rightarrow X$ sending the added edge to the concatenation $\tau\circ\sigma$ induces the inclusion $P_i^t\inj G$.
There is the following diagram of pointed spaces and the corresponding diagram of fundamental groups.
\[
\begin{tikzcd}
(Y_t)_*\arrow{r}\arrow{d}&\tilde{X}/K\arrow{d}\\
(BP_i)_*\arrow{r}{\phii}&X
\end{tikzcd}\hspace{1cm}\begin{tikzcd}
N_i^t\arrow{r}\arrow[hook]{d}[left]{\subseteq}&K\arrow[hook]{d}{\subseteq}\\
P_i^t\arrow[hook]{r}{\subseteq}&G
\end{tikzcd}
\]
In the right diagram, $\subseteq$ is used to emphasize that the maps are inclusions of subgroups of $G$ and not just injective homomorphisms.
The top horizontal map in the right diagram must be the inclusion $N_i^t\inj K$.

It remains to ensure that, under the map $(Y_t)_*\rightarrow\tilde{X}/K$ above, $Y_t$ is sent to the correct component of $\tilde{Y}/K$.
Consider the maps $\tilde{\psi}$ and $\tilde{\phii}$ induced by $\psi$ and $\phii$ on universal covers.
Let $EP_i$ denote the component of $\tilde{Y}$ contained in the image of $\tilde{\psi}$.
By construction, the image of $\tilde{\phii}$ contains the translate $t\cdot EP_i$.
So, $Y_t$ is sent to $\bar{t}\cdot EP_i/N_i$ as desired.


The maps $\left(Y_t\right)_*\rightarrow\tilde{X}/K$ define a map $f\colon V\rightarrow\tilde{X}/K$ which induces an isomorphism on fundamental groups.
By Whitehead's theorem, this is a homotopy equivalence.
\end{proof}

\section{The Main Theorem}\label{sec: homalg}

In this section, we prove \refthm{thm: Main Thm}.

\MainThm*

By \refprop{prop: ecd is finite} and \refprop{thm: LC hyp emb}, \refthm{thm: Main Thm} follows from the following, more general, result.

\begin{restatable}{theorem}{GenMainThm}\label{thm: general main thm}
Suppose $(G,\mc{P})$ is a type $F_\infty$ group pair and let $\left(\bar{G},\bar{\mc{P}}\right)$ be a type $F_\infty$ Dehn filling satisfying the Cohen-Lyndon property.
If $\op{ecd}(G,\mc{P})=n$ then there is a spectral sequence
\[
E_{pq}^2=H_p(K;H^{n-q}(G,\mc{P};\Z G))\Rightarrow H^{n-(p+q)}\left(\bar{G},\bar{\mc{P}};\Z\bar{G}\right)
\]
where $K=\ker\left(G\surj\bar{G}\right)$.
The differentials $d^r_{pq}$ have bidegree $(-1,r-1)$.
\end{restatable}

Using \refthm{thm: LC hyp emb}, we can apply \refthm{thm: general main thm} to group pairs $(G,\mc{P})$ where $\mc{P}$ is a hyperbolically embedded family of subgroups of $G$.
This gives a more general form of \refthm{thm: Main Thm}, which we record here.

\begin{restatable}{theorem}{ThmHypEmb}\label{thm: hyp emb}
Suppose $(G,\mc{P})$ is a type $F_\infty$ group pair.
Suppose also that $\mc{P}$ is a finite hyperbolically embedded family of subroups of $G$ and that $\op{ecd}(G,\mc{P})=n$.
Then the spectral sequence of \refthm{thm: general main thm} exists.
\end{restatable}

\begin{lemma}\label{lem: co F infty normal subgroups can be filled}
Let $G$ be a type $F_\infty$ group and let $N\trianglelefteq G$ be a normal subgroup such that $\bar{G}=G/N$ is also type $F_\infty$.
Let $BG$ be a classifying space for $G$ with finitely many cells in each dimension.
Then, there exists a classifying space $B\bar{G}$ for $\bar{G}$ with finitely many cells in each dimension such that $(EG)/N$ embeds $\bar{G}$ equivariantly as subcomplex of $E\bar{G}$.
\end{lemma}
\begin{proof}
Take any $B\bar{G}$ with finitely many cells in each dimension.
Then $BG\rightarrow B\bar{G}$ induces a $G$-equivariant map $EG\rightarrow E\bar{G}$.
This factors through $(EG)/N$ and the proposition follows from taking the mapping cylinder of $(EG)/N\rightarrow E\bar{G}$ as the desired model for $E\bar{G}$.
\end{proof}

\begin{convention}
For the remainder of this section, let $(G,\mc{P})$, $\left(\bar{G},\bar{\mc{P}}\right)$ and $K$ be as in the hypothesis of \refthm{thm: general main thm}.
\end{convention}

\begin{prop}\label{prop: cusped space of quotient}
$\mc{X}\left(\bar{G},\bar{\mc{P}},N\right)$ and $\mc{X}(G,\mc{P},N)/K$ can be constructed such that $\mc{X}\left(\bar{G},\bar{\mc{P}},N\right)$ contains $\mc{X}(G,\mc{P},N)/K$ as a subcomplex with complement $[0,\infty)\times V$ for some locally contractible, locally compact Hausdorff space $V$.
\end{prop}

\begin{proof}
Let $(X,Y)$ be the Eilenberg-MacLane pair used in constructing $\mc{X}(G,\mc{P})$.
If $\tilde{X}$ is the universal cover of $X$ and $\tilde{Y}$ is the preimage of $Y$ under the projection, then each component of $\tilde{Y}/K$ is $(EP_i)/(K\cap P_i)$ for some $i$.
Using \reflem{lem: co F infty normal subgroups can be filled} we can attach $E\bar{P}_i$ to $\tilde{X}/K$ along these components.
By the conclusion \ref{prop: homotopy of cusped space/K} and \cite[7.5.7]{RBrown}, the resulting space is contractible.
The proposition then follows from the construction of $\mc{X}(G,\mc{P},N)$.
\end{proof}
\begin{prop}\label{prop: cohom of quotient}
For each $N\ge0$, there is the following isomorphism.
\[
H_c^i(\mc{X}(G,\mc{P},N)/K;\Z)\cong H_c^i\left(\mc{X}\left(\bar{G},\bar{\mc{P}},N\right);\Z\right)
\]
As a consequence,
\[
H_c^i\left(\mc{X}\left(G,\mc{P}\right)/K;\Z\right)\cong H^i\left(\bar{G},\bar{\mc{P}};\Z\bar{G}\right).
\]
\end{prop}

\begin{proof}
By \refprop{prop: cusped space of quotient} and the long exact sequence of a pair for compactly supported cohomology, we get the following long exact sequence.
\[
\cdots\rightarrow H_c^i([0,\infty)\times V;\Z)\rightarrow H_c^i(\mc{X}(\bar{G},\bar{\mc{P}},N);\Z)\rightarrow H_c^i(\mc{X}(G,\mc{P},N)/K;\Z)\rightarrow\cdots
\]
By \refcor{cor: prod with half open interval}, the compactly supported cohomology of $[0,\infty)\times V$ vanishes so $H_c^i(\mc{X}(\bar{G},\bar{\mc{P}},N);\Z)\rightarrow H_c^i(\mc{X}(G,\mc{P},N)/K;\Z)$ is an isomorphism.

There are isomorphisms
\[
H_c^i(\mc{X}(G,\mc{P})/K;\Z)\cong H_c^i(\mc{X}(G,\mc{P},N)/K;\Z)\cong H_c^i\left(\mc{X}\left(\bar{G},\bar{\mc{P}},N\right);\Z\right)\cong H^i\left(\bar{G},\bar{\mc{P}};\Z\bar{G}\right)
\]
for $i\le N$.
The second statement of the proposition follows from taking $N$ arbitrarily large.
\end{proof}

We now give a spectral sequence computing $H_c^i(\mc{X}(G,\mc{P})/K;\Z)$.
Let $n=\op{ecd}(G,\mc{P})$.
For the spectral sequence, it will be convenient to regard $C_c^*(\mc{X}(G,\mc{P});\Z)$ as a chain complex bounded above such that $C_c^n(\mc{X}(G,\mc{P});\Z)$ is the degree $0$ part.
Let $F_*$ be a projective resolution of $\Z$ by $\Z K$-modules.
Then, $F_*\otimes_KC_c^{n-*}(\mc{X}(G,\mc{P});\Z)$ is a double complex with degree $m$ part $\oplus_{p+q=m}F_p\otimes_KC_c^{n-q}(\mc{X}(G,\mc{P});\Z)$.
This double complex has finitely many rows in the first quadrant but possible infinitely many rows in the fourth quadrant so there may be convergence issues with the associated spectral sequences.
We follow \cite[VII.5]{Brown} to compute the second page of the associated spectral sequences and show that this is not the case.

Consider the spectral sequence associated to the double complex where we first take the homology of the rows and then the homology of the columns.
Since each $C_c^i(\mc{X}(G,\mc{P});\Z)$ is a free $\Z K$-module\footnote{We are using that each $C_c^i(\mc{X}(G,\mc{P});\Z)$ is flat as a $\Z K$-module, which prohibits us from generalizing our theorem to study the cohomology with coefficients in $RG$, for an arbitrary commutative ring $R$.} we obtain the following computation.
\[
E^1_{pq}=\begin{cases}
C_c^{n-q}(\mc{X}(G,\mc{P});\Z)/K&p=0\\
0&p\neq0
\end{cases}
\]
The computation for the second page follows immediately.
\[
E^2_{pq}=\begin{cases}
H_c^{n-q}(\mc{X}(G,\mc{P})/K;\Z)&p=0\\
0&p\neq0
\end{cases}
\]
Thus, this spectral sequence degenerates at $E^2$ and $H_m(F_*\otimes_KC_c^{n-*}(\mc{X}(G,\mc{P});\Z))\cong H_c^{n-m}(\mc{X}(G,\mc{P});\Z)$.

Now, consider the spectral sequence obtained by taking first the homology of the columns and then homology of the rows.
Since each $F_p$ is projective, the functor $F_p\otimes_K-$ is exact and $H_m(F_p\otimes_KC_c^{n-*}(\mc{X}(G,\mc{P});\Z))\cong F_p\otimes_KH_m(C_c^{n-*}(\mc{X}(G,\mc{P});\Z))$.
The second page is given by
\[
E^2_{pq}=H_p\left(K;H_c^{n-q}\left(\mc{X}\left(G,\mc{P}\right);\Z\right)\right).
\]
By \refprop{prop: cohom contractible cusped space}, $H_c^i\left(\mc{X}\left(G,\mc{P}\right);\Z\right)\cong H^i\left(G,\mc{P};\Z G\right)$ as $\Z K$-modules so
\[
E^2_{pq}\cong H_p\left(K;H^{n-q}(G,\mc{P};\Z G)\right).
\]
By our assumption that $\op{ecd}(G,\mc{P})$ is finite, the second page has finitely many nonzero rows.
Therefore, the spectral sequence converges.
We summarize this discussion in the following lemma.

\begin{lemma}\label{lem: spectral-sequence}
There is a spectral sequence
\[
E^2_{pq}=H_p\left(K;H^{n-q}\left(G,\mc{P};\Z G\right)\right)\Rightarrow H_c^{n-(p+q)}\left(\mc{X}\left(G,\mc{P}\right)/K;\Z\right)
\]
with differentials $d^r_{pq}$ of bidegree $(-r,r-1)$.
\end{lemma}

\reflem{lem: spectral-sequence} and \refprop{prop: cohom of quotient} give \refthm{thm: general main thm}.
\bibliographystyle{alpha}
\bibliography{SSDF-Accepted.bib}

\begin{thebibliography}{DGO17}

\bibitem[BE78]{BE}
Robert Bieri and Beno Eckmann.
\newblock Relative homology and {P}oincar\'e duality for group pairs.
\newblock {\em J. Pure Appl. Algebra}, 13(3):277--319, 1978.

\bibitem[BM91]{BM91}
Mladen Bestvina and Geoffrey Mess.
\newblock The boundary of negatively curved groups.
\newblock {\em J. Amer. Math. Soc.}, 4(3):469--481, 1991.

\bibitem[Bow12]{Bow}
B.~H. Bowditch.
\newblock Relatively hyperbolic groups.
\newblock {\em Internat. J. Algebra Comput.}, 22(3):1250016, 66, 2012.

\bibitem[Bro82]{Brown}
Kenneth~S. Brown.
\newblock {\em Cohomology of groups}, volume~87 of {\em Graduate Texts in
  Mathematics}.
\newblock Springer-Verlag, New York-Berlin, 1982.

\bibitem[Bro06]{RBrown}
Ronald Brown.
\newblock {\em Topology and groupoids}.
\newblock BookSurge, LLC, Charleston, SC, 2006.

\bibitem[DGO17]{DGO}
F.~Dahmani, V.~Guirardel, and D.~Osin.
\newblock Hyperbolically embedded subgroups and rotating families in groups
  acting on hyperbolic spaces.
\newblock {\em Mem. Amer. Math. Soc.}, 245(1156):v+152, 2017.

\bibitem[GM08]{GM08}
Daniel Groves and Jason~Fox Manning.
\newblock Dehn filling in relatively hyperbolic groups.
\newblock {\em Israel J. Math.}, 168:317--429, 2008.

\bibitem[GMS19]{GMS}
Daniel Groves, Jason~Fox Manning, and Alessandro Sisto.
\newblock Boundaries of {D}ehn fillings.
\newblock {\em Geom. Topol.}, 23(6):2929--3002, 2019.

\bibitem[Hru10]{Hruska}
G.~Christopher Hruska.
\newblock Relative hyperbolicity and relative quasiconvexity for countable
  groups.
\newblock {\em Algebr. Geom. Topol.}, 10(3):1807--1856, 2010.

\bibitem[MW20]{MW}
Jason~F. Manning and Oliver~H. Wang.
\newblock Cohomology and the {B}owditch {B}oundary.
\newblock {\em Michigan Math. J.}, 69(3):633--669, 2020.

\bibitem[Osi07]{Osin07}
Denis~V. Osin.
\newblock Peripheral fillings of relatively hyperbolic groups.
\newblock {\em Invent. Math.}, 167(2):295--326, 2007.

\bibitem[Sun20]{Sun}
Bin Sun.
\newblock Cohomology of group theoretic {D}ehn fillings {I}: {C}ohen-{L}yndon
  type theorems.
\newblock {\em J. Algebra}, 542:277--307, 2020.

\end{thebibliography}
\end{document}